\documentclass[11pt]{amsart}
\usepackage{geometry}
\usepackage{graphicx}
\usepackage{amscd}
\usepackage{amssymb}
\usepackage{amsmath}
\usepackage[latin1]{inputenc}

\theoremstyle{plain}                                       %
\newtheorem{thm}{\quad Theorem}                            %
\newtheorem{prop}[thm]{\quad Proposition}                  %
\theoremstyle{definition}                                  %
\newtheorem{defi}[thm]{\quad Definition}                   %
\newtheorem{rmk}[thm]{\quad Remark}                        %
\newcommand{\N}{{\Bbb N}}
\newcommand{\K}{{\Bbb K}}

\newcommand{\al}{{\alpha}}
\newcommand{\be}{{\beta}}
\newcommand{\g}{{\gamma}}

\newcommand{\om}{{\omega}}

\title{Pascal triangle, Stirling numbers and the unique invariance of the Euler characteristic}
\author{Ana Luz\'{o}n* and Manuel A. Mor\'{o}n**}

\begin{document}
\maketitle

\address{*Departamento de Matem\'{a}tica Aplicada a los
Recursos Naturales. E.T.S.I. Montes. Universidad Polit\'{e}cnica
de Madrid. 28040-Madrid, SPAIN.}

\email{anamaria.luzon@upm.es}

\address{ **Departamento de Geometría y
Topología. Facultad de Matemáticas. Universidad Complutense de
Madrid. 28040- Madrid, SPAIN.}

 \email{mamoron@mat.ucm.es}

\vspace{1cm}

\begin{abstract}
We use some basic properties of binomial and Stirling numbers to
 prove that the Euler characteristic is, essentially, the unique
 numerical topological invariant for compact polyhedra which can be
 expressed as a linear combination of the numbers of faces of
 triangulations. We obtain this result converting it into an
 eigenvalue problem.
\end{abstract}

\vspace{1cm}

%

%



\section{Introduction}

Following D. Eppstein in \cite{Eppstein} and the introduction of
N. Levitt in \cite{Levitt}, the Euler characteristic $\chi$ is the
best known as well as the most ancient topological invariant  and
the Euler formula $E-V+F=2$ is one of many theorems in mathematics
which are important enough as to be proved repeatedly in
surprisingly many different ways.

On this line the unique invariance of the Euler characteristic,
among linear combinations on the numbers of faces of
triangulations, is known and reproved time after time. Up to our
knowledge the first proof appeared in Mayer \cite{Mayer}. More
recently in \cite{Levitt}, \cite{Forman} and \cite{Roberts2002}
there are some related results. Very recently in \cite{Yu} this
result is strengthened in the framework of combinatorial manifolds
to non-linear functions on the number of faces.

This note was born after the simple observation that the
$f$-vectors of $n$-simplices, considered as abstract simplicial
complexes, can be placed forming an infinite lower triangular
matrix which is almost the Pascal triangle. This allowed us to use
the ideas in our previous works about Riordan matrices, see for
example \cite{2ways}, \cite{teo}, \cite{BanPas} and \cite{poly}.
This kind of matrices and the group structure were introduced in
\cite{Sha91} and \cite{Spr94}.

Our main idea is to consider the $\K$-linear action induced by
such a matrix of $f$-vectors on $\K[[x]]$, where we consider the
natural $\K$-linear space structure on the set of formal power
series with coefficients on a field $\K$ of characteristic zero.
Using only this matrix we obtain, in particular, the unique
homotopy invariance of the Euler characteristic among all possible
linear combinations on components of $f$-vectors.

After that we use the known formula, involving Stirling numbers,
of the change of $f$-vectors under barycentric subdivision to
prove  that the multiples of the Euler characteristic are the
unique invariant under barycentric subdivision in the class of all
$n$-simplices. Consequently we prove the unique topological
invariance of the Euler characteristic among all possible linear
combinations on components of $f$-vectors. We obtain this result
as a consequence of the description of the eigenspace associated
to the eigenvalue $1$ in the linear action induced by an infinite
matrix describing the variation of $f$-vectors under barycentric
subdivisions.

\section{The results.}

We suppose that the basic definitions of abstract and geometric
finite simplicial complexes are known. See the corresponding
introductory chapters in \cite{Stanley} or \cite{Zielgler}. We
call herein the $f$-vector of an $n$-dimensional finite simplicial
complex $\mathcal{F}$ to $(f_0, f_1, \cdots, f_n)$, where $f_i$
counts the number of $i$-faces of $\mathcal{F}$. So $f_0$ is the
number of vertices, $f_1$ is the number of edges, and so on. We
have to note that in other places the $f$-vector includes
$f_{-1}=1$ as its first coordinate which corresponds to the
interpretation of the empty set as the unique $(-1)$-dimensional
face.

The basic pieces to construct polyhedra in Topology are the
geometric $n$-simplices, $\Delta_n$. Topologically they can be
described as the convex hull of $n+1$ affinely independent points
in a suitable euclidean space.

The abstract description of an $n$-simplex as a simplicial complex
is given by considering all the non-empty subsets of a  set of
vertices $V=\{v_0, \cdots, v_n\}$ with $n+1$ points. So, the
corresponding $f$-vectors are easily computed using combinatorial
numbers. If we denote by $f^{\Delta_n}$ the $f$-vector of
$\Delta_n$ we get
\[
f^{\Delta_n}=\left(\binom{n+1}{1},\binom{n+1}{2},\cdots\binom{n+1}{n+1}\right)
\]

We can place these vectors $f^{\Delta_n}$ forming an infinite
lower triangular matrix,
$\displaystyle{F=\left(\binom{n+1}{k+1}\right)_{n,k\geq0}}$. We
note that the non-null part coincides with Pascal's triangle
without the first row and column.


This matrix as well as the corresponding matrix representation of
Pascal's triangle, and some of its generalizations, are elements
of a group under the usual product of matrices. This group is
known as the Riordan group. We approached this group in
\cite{teo}, see also \cite{2ways}, using the Banach Fixed Point
Theorem. To describe the elements $T(\be\mid\al)$ in this group as
in \cite{teo}, we use a pair of formal power series
$\al(x)=\sum_{n\geq0}\al_nx^n$ and $\be(x)=\sum_{n\geq0}\be_nx^n$
such that $\al_0\neq0$ and $\be_0\neq0$. The columns of
$T(\be\mid\al)$ are the coefficients of the elements in the
geometric progression, in $\K[[x]]$, whose first term is the
series $\displaystyle{\frac{\be(x)}{\al(x)}}$ and the ratio is the
series $\displaystyle{\frac{x}{\al(x)}}$.

The representations of the product and the inverse in this group
are:

\[T(\be\mid \al)T(\bar{\be}\mid\bar{ \al})=T(\tilde{\be}\mid\tilde{\al})\]
where
\[
\tilde{\be}(x)=\be(x)\bar{\be} \left(\frac{x}{\al(x)}\right),
\qquad
\tilde{\al}(x)=\al(x)\bar{\al}\left(\frac{x}{\al(x)}\right)
\]

\[(T(\be\mid\al))^{-1}=T\left(\frac{1}{\be(\om^{-1})}\Big| \frac{1}{\al(\om^{-1})}\right), \qquad
\om=\frac{x}{\al}, \qquad \om\circ \om^{-1}=\om^{-1}\circ \om=x
\]
Besides, we can consider the  matrix $T(\be\mid\al)$, like in
Linear Algebra, as the associated matrix to a $\K$-linear isometry
for a suitable ultrametric $d$, see \cite{teo}, defined by:
\begin{equation}\label{E:apl}
    \begin{matrix}
T(\be \mid \al):&(\K[[x]],d)&\rightarrow&(\K[[x]],d)\\
&\g&\mapsto&T(\be\mid
\al)(\g)=\frac{\be}{\al}\g\left(\frac{x}{\al}\right)
\end{matrix}
\end{equation}

In this terms Pascal's triangle is $T(1\mid 1-x)$ and our matrix
of $f$-vectors is
$\displaystyle{F=T\left(\frac{1}{1-x}\Big|1-x\right)}$, because we
obtain $F$ from Pascal's triangle by deleting the first row and
column, see page 3614 in \cite{2ways}.

Given a $m$-dimensional simplicial complex $\mathcal{F}$ with
$f$-vector
$\displaystyle{f^{\mathcal{F}}=\left(f_0^{\mathcal{F}},f_1^{\mathcal{F}},
\cdots f_m^{\mathcal{F}}\right)}$ the Euler characteristic is
defined by
\[
\chi(\mathcal{F})=\sum_{k=0}^{m}(-1)^kf_k^{\mathcal{F}}.
\]
Consider the geometric realization $|\mathcal{F}|$ of
$\mathcal{F}$. It is known that $\chi(\mathcal{F})$ depends on
$|\mathcal{F}|$ but not on the triangulation $\mathcal{F}$. Even
more, $\chi(\mathcal{F})$ depends only on the homotopy type of
$|\mathcal{F}|$ because it can be expressed in terms of the ranks
of the homology groups $H_k(|\mathcal{F}|)$ which are homotopy
invariants. This previous result is the so called Euler-Poincaré
formula. See page 146 in \cite{Hatcher}.

Note that $\chi(\mathcal{F})$ can be expressed as the following
product of
infinite matrices : 

\begin{small}
\[
\chi(\mathcal{F})=(f_0^{\mathcal{F}},f_1^{\mathcal{F}},\cdots,f_m^{\mathcal{F}},0,\cdots)\begin{smallmatrix} \left(%
\begin{array}{c}
  1 \\
  -1 \\
  1\\
  -1 \\
  \vdots \\
   (-1)^m \\
    \vdots \\
\end{array}%
\right)
\end{smallmatrix}
\]
\end{small}
and the column matrix does not depend on the complex
$\mathcal{F}$. The generating function of this column matrix is
$\displaystyle{\frac{1}{1+x}}$ then
\[
T\left(\frac{1}{1-x}\Big|1-x\right)\left(\frac{1}{1+x}\right)=
\sum_{n\geq0}\chi(\Delta_n)x^n
\]
because the rows of the matrix $F$ are the vectors $f^{\Delta_n}$.

On the other hand, using (\ref{E:apl}) we get
\[
T\left(\frac{1}{1-x}\Big|1-x\right)\left(\frac{1}{1+x}\right)
=\frac{1}{(1-x)^2}\frac{1}{1+\frac{x}{1-x}}=\frac{1}{1-x}
\]
then
\[
\sum_{n\geq0}\chi(\Delta_n)x^n=\frac{1}{1-x} \qquad
\text{equivalently}\qquad \chi(\Delta_n)=1 \quad \forall n\geq0
\]
as everybody knows. Note that there is not any topology in the
above computation.

The Euler characteristic is, on one hand, a homotopy invariant in
the class of finite polyhedra and, on the other hand, it is a
linear combination on the number of faces of any triangulation of
the polyhedron. A natural way to define linear combinations on the
number of faces non-depending on the polyhedron even on its
dimension is the following

\begin{defi}
Let $\displaystyle{\g(x)=\sum_{n\geq0}\g_n x^n}$ be any power
series with coefficients in $\K$. Suppose that $\mathcal{F}$ is a
finite simplicial complex with $f$-vector
$(f_0^{\mathcal{F}},\cdots,f_m^{\mathcal{F}}, 0, \cdots)$, we
define the linear combination induced by the series $\g$, and
denote it by $\chi(\g,\mathcal{F})$, as
\[
\chi(\g,\mathcal{F})=\sum_{k=0}^m\g_kf_k^{\mathcal{F}}.
\]
So, the Euler characteristic is
\[
\chi\left(\frac{1}{1+x},\mathcal{F}\right).
\]
\end{defi}

We want to prove that the Euler characteristic is the unique
linear combination which is homotopy invariant. For this propose
we define

\begin{defi}
Let $\g\in\K[[x]]$, $\displaystyle{\g(x)=\sum_{n\geq0}\g_nx^n}$ be
any power series. We say that

(a) $\g(x)$ is a homotopy invariant for the class of finite
simplicial complexes if given any two of them $\mathcal{F}_1$ and
$\mathcal{F}_2$, then
$\chi(\g,\mathcal{F}_1)=\chi(\g,\mathcal{F}_2)$ provided
$|\mathcal{F}_1|$ and  $|\mathcal{F}_2|$ have the same homotopy
type.

(b) $\g(x)$ is a topological invariant for the class of finite
simplicial complexes if given any two of them $\mathcal{F}_1$ and
$\mathcal{F}_2$, then
$\chi(\g,\mathcal{F}_1)=\chi(\g,\mathcal{F}_2)$ provided
$|\mathcal{F}_1|$ and  $|\mathcal{F}_2|$ are homeomorphic.
\end{defi}

Of course, any series which is a homotopy invariant is a
topological invariant. We can restrict the above definition to any
subclass of finite simplicial complexes. Using, essentially, the
Pascal triangle we get
\begin{thm}
The unique series which are homotopy invariants for the class of
all dimensional euclidean closed balls (then for the class of all
polyhedra) are $\displaystyle{\frac{k}{1+x}}$, $k\in\K$.
\end{thm}

In other words, the multiples of the Euler characteristic are the
unique linear combinations on the components of the $f$-vectors of
finite simplicial complexes which are homotopy invariant.

\begin{proof}
Suppose a series $\displaystyle{\g(x)=\sum_{n\geq0}\g_nx^n}$ which
is a homotopy invariant for the class of closed euclidean balls.
Consider the Riordan matrix
$\displaystyle{T\left(\frac{1}{1-x}\Big|1-x\right)}$ whose rows
are specific triangulations of the euclidean balls. Then
\[
T\left(\frac{1}{1-x}\Big|1-x\right)(\g(x))=\frac{k}{1-x}
\]
for some $k\in\K$ because $\chi(\g,\Delta_n)=k$ for any $n\geq0$
from the homotopy invariance of $\g$. Now, in the Riordan group,
\[
T^{-1}\left(\frac{1}{1-x}\big|1-x\right)=T\left(\frac{1}{1+x}\Big|1+x\right)
\]
consequently
\[
\g(x)=T\left(\frac{1}{1+x}\Big|1+x\right)\left(\frac{k}{1-x}\right)=\frac{k}{1+x}
\]
and the proof is finished.
\end{proof}

\begin{rmk}
(a) The same proof is valid in the more restrictive framework of
simple homotopy theory, see \cite{Cohen}.

(b) We have really proved that the unique linear combination of
the number of faces  which assigns the same number to every
abstract $n$-simplex is the Euler characteristic. This is related
to one of the conditions imposed in \cite{Forman} because all of
them are cones.
\end{rmk}

In \cite{Brenti} the authors treat $f$-vectors of barycentric
subdivision of simplicial complexes to get, in particular, that
certain limiting behavior depending on the iteration of the
barycentric subdivision of a simplicial complex, does not depend
on the complex itself but on the dimension of such complex. In
that paper a formula for the variation of the $f$-vector after a
barycentric subdivision is given.

Now we reproduce the formula at page 850 in \cite{Brenti} taking
into account that we use the $f$-vector $(f_0,f_1,\cdots,f_m)$ and
not the extended $f$-vector $(f_{-1},f_0,f_1,\cdots,f_m)$.

\begin{prop}\label{P:bary}
Let $\mathcal{F}$ be a $m$-dimensional simplicial complex. Denote
by $sd(\mathcal{F})$ the complex obtained by the barycentric
subdivision of $\mathcal{F}$. Then
\[
f_j^{sd(\mathcal{F})}=\sum_{i=0}^mf_i^{\mathcal{F}}(j+1)!\left\{\begin{smallmatrix}i+1\\
j+1\end{smallmatrix}\right\}\qquad \text{for} \qquad j=0,\cdots,m
\]
\end{prop}
In the above result $\displaystyle{\left\{\begin{smallmatrix}k\\
l\end{smallmatrix}\right\}}$ represents the corresponding Stirling
number of the second kind as denoted in \cite{Knuth} Chapter 6.

Let  $B=(b_{i,j})_{i,j\geq0}$ be the matrix with $b_{i,j}=(j+1)!\left\{\begin{smallmatrix}i+1\\
j+1\end{smallmatrix}\right\}$ $i,j\geq0$.
The formula in the proposition above converts to
\begin{small}
\[
(f_0^{sd(\mathcal{F})},f_1^{sd(\mathcal{F})},\cdots,f_m^{sd(\mathcal{F})},0,\cdots)=(f_0^{\mathcal{F}},f_1^{\mathcal{F}},\cdots,f_m^{\mathcal{F}},0,\cdots)
\left(%
\begin{array}{ccccc}
  \left\{\begin{smallmatrix}1\\
1\end{smallmatrix}\right\}&  &  &    \\
  \left\{\begin{smallmatrix}2\\
1\end{smallmatrix}\right\} & 2!\left\{\begin{smallmatrix}2\\
2\end{smallmatrix}\right\} &  &    \\
  \vdots & \vdots & \ddots &    \\
  \left\{\begin{smallmatrix}n\\
1\end{smallmatrix}\right\} & 2!\left\{\begin{smallmatrix}n\\
2\end{smallmatrix}\right\} & \cdots & n!\left\{\begin{smallmatrix}n\\
n\end{smallmatrix}\right\}&  \\
   \vdots & \vdots & \vdots & \vdots  & \ddots \\
\end{array}%
\right)
\]
\end{small}

%
%
By the usual way $B$ induces a $\K$-linear isomorphism $B:
\K[[x]]\rightarrow\K[[x]]$, $B(\zeta(x))=\eta(x)$, such that if
\[
\zeta(x)=\sum_{n\geq0}\zeta_nx^n, \qquad \text{and} \qquad
\eta(x)=\sum_{n\geq0}\eta_nx^n
\]
then
\[
B(\zeta_n)^t=(\eta_n)^t
\]
For this operator we get
\begin{prop}
The number $1$ is an eigenvalue for the operator $B$. Moreover the
eigenspace associated to $1$ is
$\displaystyle{\left\{\frac{c}{1+x}, \ c\in\K\right\}}$.
\end{prop}

\begin{proof} If $D=(d_{i,j})$ is  the diagonal matrix  with $d_{i,i}=(i+1)!$
and $S=\left(\left\{\begin{smallmatrix}i+1\\
j+1\end{smallmatrix}\right\}\right)_{i,j\in\N}$ is the  matrix of
the Stirling numbers of second kind, then $B=SD$. Consider
\[
\delta(x)=D\left(\frac{1}{1+x}\right)=\sum_{i\geq0}(i+1)!(-x)^i
\]
then
\[
B\left(\frac{1}{1+x}\right)=\frac{1}{1+x} \quad \Leftrightarrow
\quad S(\delta)=\frac{1}{1+x}
\]
Recall that \[S^{-1}=\left((-1)^{i-j}\left[\begin{smallmatrix}i+1\\
j+1\end{smallmatrix}\right]\right)_{i,j\in\N}\] where $\displaystyle{\left[\begin{smallmatrix}i\\
j\end{smallmatrix}\right]}$ denote the Stirling numbers of the
first kind. Since
\[
\left[\begin{smallmatrix}i\\
0\end{smallmatrix}\right]=0\qquad \forall i\geq1 \qquad \text{
and} \qquad
\sum_{j=0}^i\left[\begin{smallmatrix}i\\
j\end{smallmatrix}\right]=i!
\]
Hence
\[
S^{-1}\left(\frac{1}{1+x}\right)=\delta(x)
\]
and then
\[
B\left(\frac{1}{1+x}\right)=\frac{1}{1+x}
\]
See pages 259-264 in \cite{Knuth} for the properties of Stirling
numbers used above.

 So we have proved that $1$ is an eigenvalue and that
$\displaystyle{\frac{1}{1+x}}$ is an associated eigenvector, and
then $\displaystyle{\left\{\frac{c}{1+x}, \ c\in\K\right\}}$ is
contained in the corresponding eigenspace.

We consider the finite central $(m+1)\times (m+1)$ submatrices
\[B_{m}=(b_{i,j})_{i,j=0\dots m}\] of the infinite matrix
$B$. Note that, in every $B_m$ the eigenspace associated to the
eigenvalue $1$ has always dimension $1$, because the entries in
the main diagonal are all different. To prove that there is not
any other eigenvector for $B$ associated to $1$, we suppose that
$\displaystyle{\eta(x)=\sum_{n\geq0}\eta_nx^n\neq\frac{c}{1+x}}$,
$c\in\K$ is one such eigenvector. This means that there is an
$l\in\N$, $l\geq1$ such that
\[
\sum_{j=0}^l\eta_jx^j\neq c\sum_{j=0}^l(-x)^j, \qquad \text{for
any} \qquad c\in\K.
\]
So, $\displaystyle{(\eta_j)_{j=0\cdots l}}$,
$\displaystyle{((-1)^j)_{j=0\cdots l}}$ are eigenvector associated
to $1$ for the matrix $B_l$. This is impossible because they are
linearly independent.
\end{proof}

The first consequence we obtain is that the invariance of the
\textit{Euler Characteristic by successive barycentric subdivision
is not a Topology matter}. Our proof is the following: Let
$\mathcal{F}$ be a finite simplicial complex. Denote by
$sd^{(k)}(\mathcal{F})$ the $k$-th barycentric subdivision of
$\mathcal{F}$. Using Proposition \ref{P:bary} we get
\[
f^{sd^{(k)}(\mathcal{F})}=f^{\mathcal{F}}B^k.
\]
So,
\[
\chi(sd^{(k)}(\mathcal{F}))=\chi\left(\frac{1}{1+x},sd^{(k)}(\mathcal{F})\right)=
\chi\left(B^k\left(\frac{1}{1+x}\right),\mathcal{F}\right)=
\chi\left(\frac{1}{1+x},\mathcal{F}\right)=\chi(\mathcal{F})
\]
In the above proof we only used that
$\displaystyle{\frac{1}{1+x}}$ is an eigenvector for the matrix
$B$, and then for $B^k$, associated to the eigenvalue $1$. As a
consequence of the fact that
$\displaystyle{\left\{\frac{1}{1+x}\right\}}$ is a base for the
eigenspace associated to $1$ we get

\begin{thm}
The unique linear combinations which are invariants under
barycentric subdivisions in the class of all dimensional simplices
are the multiples of the Euler characteristic. In particular,
$\displaystyle{\frac{c}{1+x}}$ are the unique series which are
topologically invariants in the class of finite simplicial
complexes.
\end{thm}

\begin{proof}
Let  $\displaystyle{\g(x)=\sum_{n\geq0}\g_nx^n}$ be a series which
is invariant under barycentric subdivisions in the class of all
dimensional simplices. This implies that
\[
T\left(\frac{1}{1-x}\Big|1-x\right)(\g)=T\left(\frac{1}{1-x}\Big|1-x\right)B(\g)
\]
this is equivalent to
\[
\g=B(\g)
\]
Consequently
\[
\g=\frac{c}{1+x}\qquad \text{ for some } \qquad c\in\K
\]
The second part follows immediately because the polyhedra
$|\mathcal{F}|$ and $|sd(\mathcal{F})|$ are always homeomorphic.
\end{proof}

\begin{rmk}
The above proof can be also applied to the more restrictive
framework of PL-Topology.
\end{rmk}

%
%
%
%
%
%
%

 {\bf Acknowledgment:} The second author was
partially supported by DGES grant MTM-2009-07030.

\end{document}